\documentclass[final,notitlepage,12pt,reqno,tbtags]{amsart}
\usepackage{graphicx}
\usepackage{calc}
\usepackage{url}
\usepackage{mathrsfs}
\usepackage{rotating}
\newlength{\depthofsumsign}
\makeatletter
\let\I\@undefined
\makeatother
\usepackage{fullpage}

\usepackage{indentfirst}
\usepackage[normalem]{ulem}
\usepackage{float}
\usepackage{amsthm}

\usepackage{enumitem}[2011/09/28]
\setenumerate{align=left, leftmargin=0pt,labelsep=.5em, labelindent=0\parindent,listparindent=\parindent,itemindent=*}
\usepackage{array}
\usepackage{empheq}
\usepackage{natbib}
\usepackage[all]{xy}

\usepackage{multirow}

\newbox\shell
\newcommand{\dia}[2]{\setbox\shell=\hbox{\begin{picture}(180,120)(-90,-60)#1
\put(-90,-60){\makebox(180,120)[b]{\large #2}}\end{picture}}\dimen0=\ht
\shell\multiply\dimen0by7\divide\dimen0by16\raise-\dimen0\box\shell\hfill}

\usepackage{caption}[2012/02/19]

\usepackage{longtable,lscape}

\usepackage{mathptmx}       
\DeclareSymbolFont{operators}{OT1}{txr}{m}{n}
\SetSymbolFont{operators}{bold}{OT1}{txr}{bx}{n}
\def\operator@font{\mathgroup\symoperators}
\DeclareSymbolFont{italic}{OT1}{txr}{m}{it}
\SetSymbolFont{italic}{bold}{OT1}{txr}{bx}{it}
\DeclareSymbolFontAlphabet{\mathrm}{operators}
\DeclareMathAlphabet{\mathbf}{OT1}{txr}{bx}{n}
\DeclareMathAlphabet{\mathit}{OT1}{txr}{m}{it}
\SetMathAlphabet{\mathit}{bold}{OT1}{txr}{bx}{it}
\DeclareSymbolFont{letters}{OML}{txmi}{m}{it}
\SetSymbolFont{letters}{bold}{OML}{txmi}{bx}{it}
\DeclareFontSubstitution{OML}{txmi}{m}{it}
\DeclareSymbolFont{lettersA}{U}{txmia}{m}{it}
\SetSymbolFont{lettersA}{bold}{U}{txmia}{bx}{it}
\DeclareFontSubstitution{U}{txmia}{m}{it}
\DeclareSymbolFontAlphabet{\mathfrak}{lettersA}
\DeclareSymbolFont{symbols}{OMS}{txsy}{m}{n}
\SetSymbolFont{symbols}{bold}{OMS}{txsy}{bx}{n}
\DeclareFontSubstitution{OMS}{txsy}{m}{n}

\usepackage{amssymb,bm,amsmath}
\usepackage{amsfonts}

\usepackage[OT2,T1,T2A]{fontenc}
\usepackage[utf8x]{inputenc}

\usepackage[english,french,german,russian]{babel}
\usepackage{appendix}

\DeclareMathOperator{\IKM}{\mathbf{IKM}}
\DeclareMathOperator{\Span}{span}

\DeclareMathOperator{\Kl}{Kl}
\DeclareMathOperator{\Sym}{Sym}

\DeclareMathOperator{\D}{d}
\DeclareMathOperator{\I}{Im}

\DeclareMathOperator{\IvKM}{\mathbf{\widetilde IKM}}
\DeclareMathOperator{\IKMh}{\mathbf{IK\widehat M}}

\DeclareMathOperator{\IKvM}{\mathbf{I\widetilde KM}}

\def\XXint#1#2#3{{\setbox0=\hbox{$#1{#2#3}{\int}$}
     \vcenter{\hbox{$#2#3$}}\kern-.5\wd0}}

\bibpunct{[}{]}{,}{n}{}{;}

\def\eor{\hfill$ \square$}

\theoremstyle{plain}
\newtheorem{theorem}{Theorem}[section]
\newtheorem{proposition}[theorem]{Proposition}
\newtheorem{lemma}[theorem]{Lemma}
\newtheorem{corollary}[theorem]{Corollary}

\newtheorem{conjecture}[theorem]{Conjecture}

\newenvironment{remark}[1][Remark]{\begin{trivlist}
\item[\hskip \labelsep {\bfseries #1}]}{\end{trivlist}}

\theoremstyle{definition}

\numberwithin{equation}{section}

\usepackage{color}
\usepackage{extarrows}

\begin{document}

\selectlanguage{english}
\title{$\mathbb Q$-linear dependence of certain Bessel moments}
\author{Yajun Zhou}
\address{Program in Applied and Computational Mathematics (PACM), Princeton University, Princeton, NJ 08544; Academy of Advanced Interdisciplinary Studies (AAIS), Peking University, Beijing 100871, P. R. China }
\email{yajunz@math.princeton.edu, yajun.zhou.1982@pku.edu.cn}

\thanks{\textit{Keywords}:  Bessel moments, Feynman diagrams\\\indent\textit{MSC 2020}:  11F03, 33C10   (Primary)  44A15, 81T18  (Secondary)\\\indent * This research was supported in part  by the Applied Mathematics Program within the Department of Energy
(DOE) Office of Advanced Scientific Computing Research (ASCR) as part of the Collaboratory on
Mathematics for Mesoscopic Modeling of Materials (CM4)}
\date{\today}

\maketitle


\begin{abstract}Let $I_0$ and $K_0$ be modified Bessel functions of the zeroth order. We use Vanhove's differential operators for Feynman integrals to derive upper bounds for  dimensions of the $\mathbb Q$-vector space spanned by certain sequences of Bessel moments \[ \left\{\left.\int_0^\infty [I_0(t)]^a[K_0(t)]^b t^{2k+1}\mathrm{d}\, t\right|k\in\mathbb Z_{\geq0}\right\},\]where $a$ and $b$ are fixed non-negative integers. For $ a\in\mathbb Z\cap[1,b)$, our upper bound for the $ \mathbb Q$-linear dimension is $\lfloor (a+b-1)/2\rfloor$, which improves the Borwein--Salvy bound  $\lfloor (a+b+1)/2\rfloor$. Our new  upper bound  $\lfloor (a+b-1)/2\rfloor$ is not sharp for $ a=2,b=6$, due to an exceptional $ \mathbb Q$-linear relation $\int_0^\infty [I_0(t)]^2[K_0(t)]^6 t\mathrm{d}\, t=72\int_0^\infty [I_0(t)]^2[K_0(t)]^6 t^{3}\mathrm{d}\, t$, which is provable by integrating modular forms.  \end{abstract}



\pagenumbering{arabic}

\section{Introduction}We define modified Bessel functions of the zeroth order as $ I_0(t)=\frac{1}{\pi}\int_0^\pi e^{t\cos\theta}\D\theta$ and $ K_0(t)=\int_0^\infty e^{-t\cosh u}\D u$, where $t>0$. For certain non-negative integers $a,b$ and $k$, the Bessel moments \begin{align}\IKM(a,b;2k+1):=
\int_0^\infty [I_0(t)]^a[K_0(t)]^b t^{2k+1}\D t
\end{align}are interesting objects to both physicists and mathematicians. In the graphical language of physicists, they represent Feynman diagrams in 2-dimensional quantum field theory \cite{BBBG2008}, and also contribute to the  finite part of renormalized perturbative expansions in $ (4-\varepsilon)$-dimensional quantum electrodynamics \cite{Laporta2008,Laporta:2017okg}. From the analytic perspective of mathematicians, every well-defined  sequence of Bessel moments $ \{\IKM(a,b;2k+1)|k\in\mathbb Z_{\geq0}\}$ is completely determined by the first few terms and a  linear recursion \cite{BorweinSalvy2007}, and certain  Bessel moments are related to critical $L$-values attached to special modular forms \cite{Broadhurst2016,Zhou2017WEF}.

According to the theory of Borwein--Salvy \cite{BorweinSalvy2007}, we have the following recursions for  sequences of Bessel moments satisfying $a+b\in\{5,6\}$ and $ a\in\mathbb Z\cap[0,b)$:\begin{align}
\sum_{j=0}^3(-1)^jp_{a+b,j}(j+k+1)\IKM(a,b;2j+k)=0
\end{align}where \cite[(11)]{BBBG2008}\begin{align}
\begin{array}{ll}
p_{5,0}(x) = x^6 & p_{6,0}(x) = x^7 \\
p_{5,1}(x) = 35 x^4 + 42 x^2 + 3 & p_{6,1}(x) =
x(56x^4+112x^2+24) \\
p_{5,2}(x) = 259 x^2 + 104 & p_{6,2}(x) = x(784x^2+944) \\
p_{5,3}(x) = 225 & p_{6,3}(x) = 2304x. \\
\end{array}
\label{crec56}
\end{align}From the recursions above, one might deduce upper bounds $ \dim_{\mathbb Q}\Span_{\mathbb Q}\{\IKM(a,b;2k+1)|k\in\mathbb Z_{\geq0}\}\leq3$ for these sequences of Bessel moments involving 5 or 6 Bessel factors in the integrands. However, for  $a+b\in\{5,6\}$ and non-vacuum diagrams satisfying $ a\in\mathbb Z\cap[1,b)$, such upper bounds are not tight enough. It has been shown by Bailey--Borwein--Broadhurst--Glasser \cite[\S5.10]{BBBG2008} that $45 \IKM(2,3;5)=228 \IKM(2,3;3)-16\IKM(2,3;1)$, and several similar sum rules have been proposed and checked up to 1200 decimal places.
\begin{conjecture}[Bailey--Borwein--Broadhurst--Glasser {\cite[\S\S5.1, 5.5, 6.1, 6.2, 6.4]{BBBG2008}}]\label{conj:BBBG_dimQ}The following integral identities are true:{\allowdisplaybreaks\begin{align}
\int_0^\infty I_0(t)[K_0(t)]^4 t(16-228t^{2}+45t^{4})\D t={}&0,\label{eq:IKM14n}\\\int_0^\infty [K_0(t)]^5 t(16-228t^{2}+45t^{4})\D t={}&24,\label{eq:IKM05n}\\\int_0^\infty [I_0(t)]^{2}[K_0(t)]^4t(2-85t^{2}+72t^{4})\D t={}&0,\label{eq:IKM24n}\\\int_0^\infty I_0(t)[K_0(t)]^5t(2-85t^{2}+72t^{4})\D t={}&0,\label{eq:IKM15n}\\\int_0^\infty [K_0(t)]^6t(2-85t^{2}+72t^{4})\D t={}&\frac{15}{2}.\label{eq:IKM06n}
\end{align}}
\end{conjecture}

In our recent work \cite[\S3]{Zhou2017WEF}, we have verified \eqref{eq:IKM14n} through explicit evaluations of $ \IKM(1,4;1)$, $ \IKM(1,4;3)$ and $\IKM(1,4;5)$. In \cite[Lemma 3.4]{Zhou2018LaportaSunrise}, we confirmed \eqref{eq:IKM24n} and \eqref{eq:IKM15n}, using a special differential operator of fourth order; a similar service was performed on \eqref{eq:IKM05n} and \eqref{eq:IKM06n} in \cite[Proposition 5.3 and Lemma 5.8]{Zhou2017BMdet}. In \S\ref{sec:dimQ_Vanhove} of this paper, we give a unified proof of all the identities in Conjecture \ref{conj:BBBG_dimQ}, along with a generalization to arbitrarily many Bessel factors,  as described in the theorem below.
\begin{theorem}[$ \mathbb Q$-linear dependence of certain Bessel moments]\label{thm:Qlin_dim}When  $ a\in\mathbb Z\cap[1,b),b\in\mathbb Z_{\geq2}$, the  set \begin{align} \{\IKM(a,b;2k+1)|k\in\mathbb Z\cap[0,(a+b-1)/2]\}\label{eq:Qlin_non_vac}\end{align}is linearly dependent over $\mathbb Q$,  and \begin{align}
\dim_{\mathbb Q}\Span_{\mathbb Q}\{\IKM(a,b;2k+1)|k\in\mathbb Z_{\geq0}\}\leq\lfloor(a+b-1)/2\rfloor.
\label{ineq:dimQ_non_vac}\tag{\ref{eq:Qlin_non_vac}$'$}\end{align}

When $n\in\mathbb Z_{\geq2}$, the  set \begin{align} \{1\}\cup\{\IKM(0,n;2k+1)|k\in\mathbb Z\cap[0,(n-1)/2]\}\label{eq:Qlin_vac}\end{align}is linearly dependent over $\mathbb Q$, and\begin{align}
\dim_{\mathbb Q}\Span_{\mathbb Q}\{\IKM(0,n;2k+1)|k\in\mathbb Z_{\geq0}\}\leq\lfloor(n+1)/2\rfloor.\label{ineq:dimQ_vac}\tag{\ref{eq:Qlin_vac}$'$}
\end{align}

In the statements above, $ \lfloor x\rfloor$ stands for the greatest integer less than or equal to $x$.\end{theorem}

Here, we point out that the inequality in \eqref{ineq:dimQ_non_vac} follows immediately from the $ \mathbb Q$-linear dependence of the set in \eqref{eq:Qlin_non_vac}, because the Borwein--Salvy theory  \cite[Theorem 1.1]{BorweinSalvy2007} has already provided us with a linear recursion with non-vanishing integer coefficients, involving $\lfloor(a+b+3)/2\rfloor $ consecutive terms in the corresponding sequence  $ \{\IKM(a,b;2k+1)|k\in\mathbb Z_{\geq0}\}$.  Thus, our upper bound  given in \eqref{ineq:dimQ_non_vac} is exactly one dimension smaller than  what is inferrable from the  Borwein--Salvy recurrence.
Meanwhile, the inequality in   \eqref{ineq:dimQ_vac} gives the same upper bound on $\mathbb Q$-linear dimension as the Borwein--Salvy recursion.
What is new here is that the  $\mathbb Q$-linear basis can be constructed from a subset of  $ \{1\}\cup\{\IKM(0,n;2k+1)|k\in\mathbb Z\cap[0,(n-3)/2]\}$, rather than a subset of vacuum diagrams $ \{\IKM(0,n;2k+1)|k\in\mathbb Z\cap[0,(n-1)/2]\}$. Recently, a more insightful interpretation of the dimension bounds in Theorem \ref{thm:Qlin_dim}, based on de Rham cohomology classes, has been presented by Fres\'an--Sabbah--Yu \cite[Remark 6.8]{FresanSabbahYu2020b}.

The present author initially had thought that the bound in  \eqref{ineq:dimQ_non_vac} could no longer be improved, until a numerical  counterexample $\IKM(2,6;1)=72\IKM(2,6;3) $ suggested otherwise. Extending some modular techniques developed in \cite{Zhou2017WEF}, we will  verify this surprising $ \mathbb Q$-linear dependence  in \S\ref{sec:IKM261IKM263}, along with a few related results  in the theorem below.

\begin{theorem}[Some exceptional sum rules]\label{thm:8Bessel_surprise} If we define \begin{align}\IKMh(4,4;3):=
\int_0^\infty[I_0(t)K_0(t)]^2\left\{ [I_0(t)K_0(t)]^2-\frac{1}{4t^{2}} \right\}t^{3}\D t
\end{align}as an ``honorary Bessel moment'', then we have the following identities:\begin{align}\IKM(4,4;1)-72\IKMh(4,4;3)={}&\frac{7\log2}{2} ,\label{eq:IKM441IKM*443}\\\IKM(3,5;1)-72\IKM(3,5;3)={}&-\frac{5 \pi ^2}{12} ,\label{eq:IKM351IKM353}\\
\IKM(2,6;1)-72\IKM(2,6;3)={}&0,\label{eq:IKM261IKM263}\\\IKM(1,7;1)-72\IKM(1,7;3)={}&\frac{7\pi^{4}}{48}.\label{eq:IKM171IKM173}
\end{align}Moreover, the set $ \{\pi,\IKM(1,7;1),\IKM(1,7;5),\IKM(2,6;1),\IKM(2,6;5)\}$ is algebraically dependent, under the constraint of a non-linear sum rule\begin{align}
7 \pi ^4 \IKM(2,6;1)-6912[\IKM(1,7;1) \IKM(2,6;5)- \IKM(1,7;5) \IKM(2,6;1)]=\frac{45 \pi ^6}{16}.\label{eq:nonlin_sum_8Bessel}
\end{align}\end{theorem}

We note that the algebraic relations [over $ \mathbb Q(\pi,\log 2)$] displayed in \eqref{eq:IKM441IKM*443}--\eqref{eq:nonlin_sum_8Bessel} are truly aberrant, in that they do not belong to the Broadhurst--Roberts ideal \cite[\S4.2]{Zhou2020BRquad} that presumably\ exhausts algebraic relations for on-shell Bessel moments.

\noindent\textbf{Acknowledgments.} The author thanks David Broadhurst and  Javier Fres\'an for their perceptive comments on the initial draft of this manuscript. The author is grateful to the anonymous referee for suggestions on improving the presentations of this work. \section{Applications of Vanhove's differential equations to $ \mathbb Q$-linear dependence\label{sec:dimQ_Vanhove}}

As in our recent work \cite[\S2]{Zhou2017BMdet}, we introduce abbreviations for off-shell Feynman diagrams \begin{align}
\IvKM(a+1,b;m|u):={}&\int_0^\infty I_{0}(\sqrt{u}t)[I_0(t)]^a[K_0(t)]^{b}t^{m}\D t,\\\IKvM(a,b+1;m|u):={}&\int_0^\infty K_{0}(\sqrt{u}t)[I_0(t)]^a[K_0(t)]^{b}t^{m}\D t,
\end{align} whenever  the non-negative integers $ a,b,m\in\mathbb Z_{\geq0}$ ensure convergence of the integrals. For a smooth bivariate function $ f(t,u)$, we define $ D^0f(t,u)=f(t,u)$, and $ D^{n+1}f(t,u)=\frac{\partial}{\partial u}D^nf(t,u)$ for all $ n\in\mathbb Z_{\geq0}$. For convenience, we will also write $ D^n f(t,1)$ for evaluating $ D^nf(t,u)$ at $u=1$, and so forth.

Vanhove's operator $ \widetilde L_n$ is an $n$th order holonomic differential operator in the variable $u$, which acts on off-shell Feynman diagrams $ \IvKM(a,n+2-a;1|u),a\in\mathbb Z\cap[1,(n+2)/2)$ to produce constants. For example, Vanhove's third- and fourth-order operator can be written explicitly as follows  \cite[Table 1, $n=4$ and $n=5$]{Vanhove2014Survey}: \begin{align}\begin{split}
\widetilde {L}_3:={}&u^{2}(u-4)(u-16)D^{3}+6u(u^2-15u+32)D^{2}\\{}&+(7u^2-68u+4)D^{1}+(u-4)D^{0},\end{split}
\\\begin{split}\widetilde L_4:={}&
u^2(u-25) (u-9) (u-1) D^{4}+2 u (5 u^3-140 u^2+777 u-450)D^3\\{}&+(25 u^3-518 u^2+1839 u-450)D^{2}+(3 u-5) (5 u-57)D^{1}+(u-5)D^{0}.\end{split}\label{eq:VL4}\end{align}They satisfy the following differential equations of Vanhove's type \cite[Lemmata 2.2 and 3.1]{Zhou2017BMdet}: \begin{align}&\begin{cases}
\widetilde L_3\IvKM(2,3;1|u)= 0,& \forall u\in(0,4);\ \ \\
\widetilde L_{3}\IvKM(1,4;1|u)=-3, & \forall u\in(0,16); \\\widetilde L_3\IKvM(1,4;1|u)=\frac{3}{4}, & \forall u\in(0,\infty).
\end{cases}
\\&
\begin{cases}\widetilde L_4\IvKM(2,4;1|u)= 0,& \forall u\in(0,9);\\
\widetilde L_4\IvKM(1,5;1|u)= -\frac{15}{2}, & \forall u\in(0,25); \\\widetilde L_4\IKvM(1,5;1|u)=\frac{3}{2}, & \forall u\in(0,\infty).
\end{cases}
\end{align}

In what follows, we use Vanhove's differential equations to produce an algorithmic proof of  Theorem \ref{thm:Qlin_dim}, which also recovers Conjecture \ref{conj:BBBG_dimQ} as special cases.
\begin{theorem}[Sum rules for arbitrarily many Bessel factors]For each $ n\in\mathbb Z_{>0}$, there exists a non-zero polynomial    $f_n(\xi)\in\frac{1}{n+4}\mathbb Z[\xi]$ whose degree  does not exceed $\lfloor (n+1)/2\rfloor  $, such that   the following  inhomogeneous sum rule  \begin{align}
\int_0^\infty[K_0(t)]^{n+2}tf_n(t^2)\D t=(n+1)!\equiv\Gamma(n+2)\label{eq:vac_sum_rule}
\end{align}holds. Accordingly, the same polynomial applies to a homogeneous sum rule\begin{align}
\int_0^\infty [I_{0}(t)]^{a}[K_0(t)]^{n+2-a}tf_n(t^2)\D t=0,\label{eq:nonvac_sum_rule}
\end{align}where $ a\in\mathbb Z\cap[1,(n+2)/2)$.

As a result, the statements about $ \mathbb Q$-linear dependence in Theorem \ref{thm:Qlin_dim} are true. \end{theorem}\begin{proof}We recall from \cite[Lemma 4.2]{Zhou2017BMdet} that we have the following differential equations of Vanhove's type:\begin{align}
\begin{cases}\widetilde L_n\IvKM(1,n+1,1|u)=-\frac{(n+1)!}{2^{n}}, &  \\
\widetilde L_n\IKvM(1,n+1,1|u)=\frac{n!}{2^{n}}, &  \\\widetilde
L_n\IvKM(j,n+2-j,1|u)=0, & \forall j\in\mathbb Z\cap[2,\frac{n}{2}+1], \\ \widetilde L_n\IKvM(j,n+2-j,1|u)=0,&\forall j\in\mathbb Z\cap[2,\frac{n+1}{2}],
\end{cases}\label{eq:VanhoveLnODE}
\end{align} which can be verified through  the relations \begin{align}
\left\{ \begin{array}{c}
t\widetilde L_nI_0(\sqrt{u}t)=\frac{(-1)^n}{2^n}L^*_{n+2}\frac{I_0(\sqrt{u}t)}{t}, \\
t\widetilde L_nK_0(\sqrt{u}t)=\frac{(-1)^n}{2^n}L^*_{n+2}\frac{K_0(\sqrt{u}t)}{t}, \\
\end{array} \right.
\end{align}and integrations by parts in the variable $t$. Here,\begin{align}
L^*_{n+2}g(t,u)=\sum_{k=0}^{n+2}(-1)^{k}\frac{\partial^k}{\partial t^k}[\lambda_{n+2,k}(t)g(t,u)],
\label{eq:form_adj}\end{align} defines the formal adjoint to the Borwein--Salvy operator \begin{align}L_{n+2}=\mathscr L_{n+2,n+2}=\sum_{k=0}^{n+2}\lambda_{n+2,k}(t)\frac{\partial^k}{\partial t^k},\end{align}which in turn, is constructed   by the Bronstein--Mulders--Weil algorithm \cite[Theorem 1]{BMW1997}:\begin{align}
\begin{cases}\mathscr L_{n+2,0}=\left( t\frac{\D }{\D t} \right)^0,\mathscr L_{n+2,1}=t\frac{\D }{\D t}, &  \\
\mathscr L_{n+2,k+1}=t\frac{\D }{\D t}\mathscr L_{n+2,k}-k(n+2-k)t^{2}\mathscr L_{n+2,k-1}, &          \forall k\in\mathbb Z\cap[1,n+1]. \\
\end{cases}\label{eq:BMW}
\end{align}We also remind our readers of the fact that the Borwein--Salvy operator  $ L_{n+2}$ is the $(n+1)$st symmetric power of the Bessel differential operator, so that it annihilates every member in the set $ \{[I_0(t)]^a[K_0(t)]^{n+1-a}|a\in\mathbb Z\cap[0,n+1]\}$.

As we specialize the procedure of integration by parts in \cite[(4.24)]{Zhou2017BMdet} to $ u=1$, we have\begin{align}\begin{split}
0={}&\int_0^\infty \frac{I_0(t)}{t}L_{n+1}\{[K_0(t)]^n\}\D t=\int_0^\infty \frac{I_0(t)}{t}\mathscr L_{n+1,n+1}\{[K_0(t)]^n\}\D t\\={}&\int_0^\infty I_0(t)\frac{\D}{\D t}\mathscr L_{n+1,n}\{[K_0(t)]^n\}\D t-n\int_0^\infty tI_0(t)\mathscr L_{n+1,n-1}\{[K_0(t)]^n\}\D t\\={}&-(-1)^{n}n!-\int_0^\infty\mathscr  L_{n+1,n}\{[K_0(t)]^n\}\frac{\D I_0(t)}{\D t}\D t\\{}&-n\int_0^\infty tI_0(t)\mathscr L_{n+1,n-1}\{[K_0(t)]^n\}\D  t.\end{split}\label{eq:int_part_sunrise}
\end{align} From \cite[Lemma 4.2]{Zhou2017BMdet}, we know that all subsequent integrations by parts will not result in any boundary contributions like $ (-1)^{n}n!$. Without loss of generality, we assume that $ n\in\mathbb Z_{\geq2}$, and carry on the computations above a few steps further:{\allowdisplaybreaks\begin{align}\begin{split}
(-1)^{n}n!={}&-\int_0^\infty I_1(t)\mathscr  L_{n+1,n}\{[K_0(t)]^n\}\D t-n\int_0^\infty tI_0(t)\mathscr L_{n+1,n-1}\{[K_0(t)]^n\}\D  t\\={}&\int_0^\infty\left\{ \frac{\D [tI_1(t)]}{\D t}-ntI_0(t) \right\}\mathscr L_{n+1,n-1}\{[K_0(t)]^n\}\D  t\\{}&+2(n-1)\int_0^\infty t^{2}I_1(t)\mathscr  L_{n+1,n-2}\{[K_0(t)]^n\}\D t\\={}&-(n-1)\int_0^\infty tI_0(t)\mathscr L_{n+1,n-1}\{[K_0(t)]^n\}\D  t\\{}&+2(n-1)\int_0^\infty t^{2}I_1(t)\mathscr  L_{n+1,n-2}\{[K_0(t)]^n\}\D t.\end{split}
\end{align}} Arguing along this line, and exploiting the following identities for $ m\in\mathbb Z_{\geq0}$:\begin{align}
\frac{\D}{\D t}[t^{2 m+1} I_1(t)]={}&2 m t^{2 m} I_1(t)+t^{2 m+1} I_0(t),\\\frac{\D}{\D t}[t^{2 m} I_0(t)]={}&t^{2 m} I_1(t)+2 m t^{2 m-1} I_0(t),
\end{align} we have \begin{align}\begin{split}
(-1)^{n}n!={}&\int_0^\infty [K_0(t)]^nL_{n+1}^*\frac{I_0(t)}{t}\D t\\={}&(n-1)\int_0^\infty [tA_{n,\lfloor(n-1)/2\rfloor}(t^2)I_0(t)+t^{2}B_{n,\lfloor (n-2)/2\rfloor}(t^2)I_1(t)][K_0(t)]^n\D  t\end{split}\label{eq:A+B}
\end{align}where $ A_{n,\lfloor(n-1)/2\rfloor}$ (resp.~$B_{n,\lfloor (n-2)/2\rfloor}$) is a  polynomial with integer coefficients, whose degree does not exceed $\lfloor(n-1)/2\rfloor $ (resp.~$\lfloor(n-2)/2\rfloor$).

By a similar procedure, with the following identities for $ m\in\mathbb Z_{\geq0}$:\begin{align}
\frac{\D}{\D t}[t^{2 m+1} K_1(t)]={}&2 m t^{2 m} K_1(t)-t^{2 m+1} K_0(t),\\\frac{\D}{\D t}[t^{2 m} K_0(t)]={}&-t^{2 m} K_1(t)+2 m t^{2 m-1} K_0(t),
\end{align} we can deduce from \cite[(4.26)]{Zhou2017BMdet} the following equation:
 \begin{align}\begin{split}
-(-1)^{n}(n-1)!={}&\int_0^\infty I_0(t)[K_0(t)]^{n-1}L_{n+1}^*\frac{K_0(t)}{t}\D t\\={}&(n-1)\int_0^\infty [tA_{n,\lfloor(n-1)/2\rfloor}(t^2)K_0(t)-t^{2}B_{n,\lfloor (n-2)/2\rfloor}(t^2)K_1(t)]I_0(t)[K_0(t)]^{n-1}\D  t.\end{split}\label{eq:A-B}
\end{align}

Bearing in mind the Wro\'nskian relation $I_0(t)K_1(t)+I_1(t)K_0(t)=\frac{1}{t} $, we may subtract \eqref{eq:A-B} from \eqref{eq:A+B}, arriving at\begin{align}
(-1)^{n}(n-2)!(n+1)=\int_0^\infty tB_{n,\lfloor (n-2)/2\rfloor}(t^2)[K_0(t)]^{n-1}\D  t.
\end{align}Choosing $ f_{n-3}(\xi)=\frac{(-1)^{n}}{n+1}B_{n,\lfloor (n-2)/2\rfloor}(\xi)\in\frac{1}{n+1}\mathbb Z[\xi]$, we can verify the first sentence in our theorem.

So far, we have verified the inhomogeneous sum rules for vacuum diagrams in any loop order. With the last two lines in \eqref{eq:VanhoveLnODE}, we can establish the homogeneous sum rules for non-vacuum diagrams in a similar vein, if not simpler.   \end{proof}
\begin{remark}
As a service to the quantum field community, we list our computations of $ f_n(t^2),n\in\mathbb Z\cap[1,10]$ in Table \ref{tab:fn_examples}. Clearly, the entries  $ f_3$ and $f_4$ allow us to verify all the identities declared in Conjecture \ref{conj:BBBG_dimQ}. (It is our hope that, by working a little harder, one can perhaps show that $ f_n(\xi)\in\mathbb Z[\xi]$ and $ f_n(0)=2^{n+1}$ for all $ n\in\mathbb Z_{>0}$. However, we are not going to pursue in this direction, as it will not affect the qualitative structure of $\mathbb Q$-linear dependence.)
\begin{table}[ht]
{\caption{The first ten polynomials  that satisfy  the  inhomogeneous sum rules \eqref{eq:vac_sum_rule} }\label{tab:fn_examples}}\begin{small}\begin{align*}\begin{array}{c|l}\hline\hline  n&\vphantom{\frac\int\int}f_n(t^2)\\\hline \vphantom{\frac{\int}{1}}1&4-3t^2\\2&8(1-4t^2)\\3&16-228t^{2}+45t^{4}\\4&16(2-85t^{2}+72t^{4})\\5&64 - 7344 t^2 + 17720 t^4 - 1575 t^6\\6&128 (1- 291 t^2 + 1662 t^4 -576 t^6)\\7&256 - 181056 t^2 + 2199408 t^4 - 1974168 t^6 + 99225 t^8\\8&256 (2 - 3335 t^2 + 80370 t^4 - 155256 t^6 + 28800 t^8)\\9&1024 - 3936000 t^2 + 179222016 t^4 - 669169296 t^6 + 304572636 t^8 -
 9823275 t^{10}\\10&2048 (1 - 8708 t^2 + 722853 t^4 -4861164 t^6 + 4513680 t^8 -518400 t^{10})\\\hline\hline\end{array}\end{align*}\end{small}\end{table}
\eor\end{remark}\begin{remark}Let $ C=\frac{1}{240 \sqrt{5}\pi^{2}}\Gamma \left(\frac{1}{15}\right) \Gamma \left(\frac{2}{15}\right) \Gamma \left(\frac{4}{15}\right) \Gamma \left(\frac{8}{15}\right)$ be the ``Bologna constant'' attributed to Broadhurst \cite{Broadhurst2007,BBBG2008} and Laporta \cite{Laporta2008}. The results  $\IKM(2,3;1)=\frac{\sqrt{15}\pi}{2}C $, $ \IKM(2,3;3)=\frac{\sqrt{15}\pi}{2}\left( \frac{2}{15} \right)^2\big( 13C+\frac{1}{10C} \big)$ and $ \IKM(2,3;5)=\frac{\sqrt{15}\pi}{2}\left( \frac{4}{15} \right)^3\big( 43C+\frac{19}{40C} \big)$ were confirmed by Bai\-ley--Bor\-wein--Broad\-hurst--Glas\-ser \cite[\S5.10]{BBBG2008}. In \cite[\S2]{Zhou2017WEF}, we wrote a short proof for $ \IKM(1,4;1)=\pi^2 C$ based on Wick rotations, which simplified the arguments by  Bloch--Kerr--Vanhove \cite{BlochKerrVanhove2015} and Samart \cite{Samart2016}; we then verified the formulae $ \IKM(1,4;3)=\pi^{2}\left( \frac{2}{15} \right)^2\left( 13C-\frac{1}{10C} \right)$ and $ \IKM(1,4;5)=\pi^{2}\left( \frac{4}{15} \right)^3\left( 43C-\frac{19}{40C} \right)$ using Eichler integrals \cite[\S3]{Zhou2017WEF}. In retrospect, we could have dispensed with the computations of Eichler integrals, in favor of some algebraic manipulations related to Vanhove's operators. Indeed, in \cite[\S2]{Zhou2017BMdet}, we demonstrated that the determinant\begin{align}\det
\begin{pmatrix}\IKM(1,4;1) & \IKM(1,4;3) \\
\IKM(2,3;1) & \IKM(2,3;3) \\
\end{pmatrix}=\frac{2\pi^{3}}{\sqrt{3^35^5}}
\end{align}followed from factorizations of certain Wro\'nskians related to  $ \widetilde L_3$, which would enable us to evaluate $ \IKM(1,4;3)$ algebraically, drawing on the knowledge of the other three matrix elements. The sum rule   \eqref{eq:IKM14n},  provable by the Vanhove procedure outlined above,  now provides us with an algebraic route towards the closed form of $ \IKM(1,4;5)$.       \eor\end{remark}

\section{Hankel--Vanhove mechanism and exceptional sum rules\label{sec:IKM261IKM263}}To prove Theorem~\ref{thm:8Bessel_surprise}, we will need to investigate certain modular forms on the  Chan--Zudilin group $\varGamma_0(6)_{+3}=\langle\varGamma_0(6),\widehat W_3\rangle$  \cite{ChanZudilin2010}, which is \begin{align}
\varGamma_0(6):=\left\{ \left.\begin{pmatrix}a & b \\
c & d \\
\end{pmatrix}\right|a,b,c,d\in\mathbb Z;ad-bc=1;c \equiv0\,(\!\bmod 6)\right\}
\end{align}adjoining an involution  $ \widehat W_3:=\frac{1}{\sqrt{3}}\left(\begin{smallmatrix}3&-2\\6&-3\end{smallmatrix}\right)$. In the meantime, we will revisit our treatment of the 8-Bessel problems in \cite[\S5]{Zhou2017WEF}, with some extensions and simplifications.
\subsection{Chan--Zudilin representations of Bessel moments and their Hankel fusions}
As pointed out by    Chan--Zudilin \cite[(2.2)]{ChanZudilin2010}, the group   $ \varGamma_0(6)_{+3}$ enjoys a Hauptmodul \begin{align}
X_{6,3}(z):={}&\left[\frac{ \eta (2z ) \eta (6 z )}{\eta (z ) \eta (3z)}\right]^{6},\quad z\in\mathfrak H:=\{\tau\in\mathbb C|\I\tau>0\}, \label{eq:X63_defn}
\end{align} expressible  in terms of  the Dedekind eta function  $ \eta(\tau):=e^{\pi i\tau/12}\prod_{n=1}^\infty(1-e^{2\pi in\tau}),\tau\in\mathfrak H$. This Hauptmodul satisfies the following invariance properties:\begin{align}
X_{6,3}(\widehat\gamma z)={}&X_{6,3}(z),&& \forall z\in\mathfrak H, \widehat \gamma\in\varGamma_0(6),\\X_{6,3}(\widehat W_3 z)={}&X_{6,3}(z),&& \forall z\in\mathfrak H,\label{eq:X63W3}
\end{align}where we have set  $ \widehat T z:=\frac{az+b}{cz+d}$ for $ \widehat T=\left(\begin{smallmatrix}a&b\\c&d\end{smallmatrix}\right)$, by convention.

Moreover, in \cite[(2.5)]{ChanZudilin2010}, Chan--Zudilin introduced a notation \begin{align}
Z_{6,3}(z):={}&\frac{[\eta(z)\eta(3z)]^{4}}{[\eta(2z)\eta(6z)]^{2}},\quad z\in\mathfrak H  \label{eq:Z63_defn}
\end{align}      for a  modular form of weight 2 on $\varGamma_0(6)_{+3} $.
It transforms  as follows:\begin{align}
Z_{6,3}(\widehat\gamma z)={}&(cz+d)^{2}Z_{6,3}(z),&& \forall z\in\mathfrak H, \widehat \gamma=\left(\begin{smallmatrix}a&b\\c&d\end{smallmatrix}\right)\in\varGamma_0(6),\\Z_{6,3}(\widehat W_3 z)={}&-3 (2 z-1)^2Z_{6,3}(z),&& \forall z\in\mathfrak H.\label{eq:Z63W3}
\end{align}

The modular form  $ Z_{6,3}(z)$ and the Hauptmodul  $ X_{6,3}(z)$  are both useful in the evaluation of Bessel moments, as shown by the proposition below.\begin{proposition}[Modular parametrization of Bessel moments]We have \begin{align}\begin{split}&
\IvKM(2,3;1|-64X_{6,3}(z)):=\int_0^\infty I_{0}\left(\left[\frac{2 \eta (2z ) \eta (6 z )}{\eta (z ) \eta (3z)}\right]^{3}\frac{t}{i}\right)I_0(t)[K_0(t)]^3t\D t\\={}&\frac{\pi^{2}}{16}Z_{6,3}(z),\label{eq:IvKM_Z63}\end{split}\\\begin{split}&
\IKvM(2,3;1|-64X_{6,3}(z)):=\int_0^\infty K_{0}\left(\left[\frac{2 \eta (2z ) \eta (6 z )}{\eta (z ) \eta (3z)}\right]^{3}\frac{t}{i}\right)[I_0(t)K_0(t)]^2t\D t\\={}&\frac{\pi^{2}}{96}[1-3(2z-1)^{2}]Z_{6,3}(z),\label{eq:IKvM_Z63}\end{split}\\\begin{split}&\IvKM(1,4;1|-64X_{6,3}(z))+4\IKvM(1,4;1|-64X_{6,3}(z))\\={}&\int_0^\infty I_{0}\left(\left[\frac{2 \eta (2z ) \eta (6 z )}{\eta (z ) \eta (3z)}\right]^{3}\frac{t}{i}\right)[K_0(t)]^4t\D t+4\int_0^\infty K_{0}\left(\left[\frac{2 \eta (2z ) \eta (6 z )}{\eta (z ) \eta (3z)}\right]^{3}\frac{t}{i}\right)I_0(t)[K_0(t)]^3t\D t\\={}&\frac{\pi^3}{8i}(2z-1)Z_{6,3}(z),\end{split}
\label{eq:IvKM_IKvM_sum}\end{align}for   $ z=\frac{1}{2}+iy,y\geq\frac{1}{2\sqrt{3}}$, where $ u=-64X_{6,3}(z)\in(0,4]$. For $ u=-64X_{6,3}(z)\in[4,\infty)$, the relation in  \eqref{eq:IKvM_Z63} remains a valid modular parametrization for  the  Bessel moment $\int_0^\infty K_{0}(\sqrt{u}t)[I_0(t)K_0(t)]^2t\D t $, where\begin{align}
\begin{cases}u=-64X_{6,3}(z)\in[4,16], & z=\frac{1}{2}+\frac{i}{2\sqrt{3}}e^{i\varphi},\varphi\in[0,\frac{\pi}{3}], \\
u=-64X_{6,3}(z)\in[16,\infty), & z=\frac{1}{6}(1+e^{i\psi}),\psi\in[\frac{\pi}{3},\pi). \\
\end{cases}
\end{align}\end{proposition}\begin{proof}Both \eqref{eq:IvKM_Z63} and \eqref{eq:IKvM_Z63} can be verified in three steps. First, using integration by parts and symmetric powers of Bessel differential operators \cite[Lemma 4.2]{Zhou2017BMdet}, one checks that  $ f(u)=\IvKM(2,3;1|u)$  [or    $ f(u)=\IKvM(2,3;1|u)$]    satisfies a homogeneous differential equation \begin{align} u^2 (u - 4) (u - 16)f'''(u)+6 u (u^2 - 15 u + 32)f''(u)+(7 u^2 - 68 u + 64)f'(u)+(u-4)f(u)=0.\label{eq:L3_ODE}\end{align}Second, one notes that every solution to such a homogeneous differential equation must assume the form  $ f(-64X_{6,3}(z))=Z_{6,3}(z)(c_{0}+c_1z+c_2 z^2)$ for constants $c_0,c_1,c_2\in\mathbb C $ \cite[Theorems 1 and 3]{Verrill1996}.
Third, examining special values (including asymptotic behavior) of the function in question, one  determines the constants $c_0,c_1,c_2$.

Here, the  third step deserves further explanations.

For  \eqref{eq:IvKM_Z63} (see also \cite[(3.1.6)]{Zhou2017WEF}), we have   the following asymptotic behavior \cite[(54)]{BBBG2008}:\begin{align}
\lim_{z\to\frac{1}{2}+i\infty}\IvKM(2,3;1|-64X_{6,3}(z))=\int_0^\infty I_0(t)[K_0(t)]^3t\D t=\frac{\pi^{2}}{16}.
\end{align}Meanwhile, the relation  $ \lim_{z\to\frac{1}{2}+i\infty}Z_{6,3}(z)=1$ follows from \eqref{eq:Z63_defn} and the definition of the Dedekind eta function as an infinite product.
 Thus one  immediately determines $ c_0=\frac{\pi^2}{16},c_1=c_2=0$.

For  \eqref{eq:IKvM_Z63}, we need  two observations: (i)~The function $ \IKvM(2,3;1|-64X_{6,3}(z))$ analytically continues across the point $ z=\frac{1}{2}+\frac{i}{2\sqrt{3}}$, and remains holomorphic in an open neighborhood of the ray $ z=\frac{1}{2}+iy,y>0$; and (ii)~To remain compatible with the transformation laws in  \eqref{eq:X63W3} and \eqref{eq:Z63W3}, we must have \begin{align}
\IKvM(2,3;1|-64X_{6,3}(z)):=\{k_{0}[1-3(2z-1)^{2}]+k_{1}(2z-1)\}Z_{6,3}(z)
\end{align} on the ray $ z=\frac{1}{2}+iy,y>0$. To compute the constants $ k_0,k_1\in\mathbb C$, we rely on  two quick facts: (1)~We have  $ \IKvM(2,3;1|u)=\frac{1}{32}\log^2\frac{4}{u}+O(\log u)$, as $u\to0^+$ \cite[(3.18)--(3.19)]{Zhou2018LaportaSunrise}, so $k_0= \frac{\pi^{2}}{96}$; (2)~We have $ \IKvM\big(2,3;1\big|-64X_{6,3}\big(\frac{1}{2}+\frac{i\sqrt{5}}{2\sqrt{3}}\big)\big)=\IKvM(2,3;1|1)=\IKM(2,3;1)=\IvKM\big(2,3;1\big|-64X_{6,3}\big(\frac{1}{2}+\frac{i\sqrt{5}}{2\sqrt{3}}\big)\big)=\frac{\pi^{2}}{16}Z_{6,3}\big(\frac{1}{2}+\frac{i\sqrt{5}}{2\sqrt{3}}\big)$, as can be inferred from  \cite[Table 1]{Zhou2017WEF} and \eqref{eq:IvKM_Z63}, so $k_1=0 $.

The proof of \eqref{eq:IvKM_IKvM_sum} (which corrects a misprinted sign in \cite[(5.1.33)]{Zhou2017WEF}) is similar to that of  \eqref{eq:IKvM_Z63}. See \cite[Proposition 5.1.4]{Zhou2017WEF} for details.     \end{proof}\begin{corollary}[Analytic continuations of Bessel moments]Let   $ J_0(x):=\frac{2}{\pi}\int_0^{\pi/2}\cos(x\cos\varphi)\D\varphi\equiv I_0(ix),x\in\mathbb R$ and  $ Y_0(x):=-\frac{2}{\pi}\int_0^\infty\cos(x\cosh u)\D u,x\in(0,\infty)$ be Bessel functions of the zeroth order.   For $ z/i>0$, the following identities hold:\begin{align}
\int_0^\infty J_{0}\left(\left[\frac{2 \eta (2z ) \eta (6 z )}{\eta (z ) \eta (3z)}\right]^{3}t\right)I_0(t)[K_0(t)]^3t\D t={}&\frac{\pi^{2}}{16}Z_{6,3}(z),\label{eq:JIKKK_Z63}\\
\int_0^\infty J_{0}\left(\left[\frac{2 \eta (2z ) \eta (6 z )}{\eta (z ) \eta (3z)}\right]^{3}t\right)[I_{0}(t)K_0(t)]^{2}t\D t={}&\frac{\pi z }{4i}Z_{6,3}(z),\label{eq:J_zZ63}\\\int_0^\infty Y_{0}\left(\left[\frac{2 \eta (2z ) \eta (6 z )}{\eta (z ) \eta (3z)}\right]^{3}t\right)[I_{0}(t)K_0(t)]^{2}t\D t={}&\frac{\pi\left( z^2+\frac{1}{6} \right)}{4}Z_{6,3}(z).\label{eq:Y_zzZ63}
\end{align}Furthermore, we have the following relation for $ z/i>0$:\begin{align}\begin{split}
&\int_0^\infty J_{0}\left(\left[\frac{2 \eta (2z ) \eta (6 z )}{\eta (z ) \eta (3z)}\right]^{3}t\right)[K_0(t)]^4t\D t-2\pi\int_0^\infty Y_{0}\left(\left[\frac{2 \eta (2z ) \eta (6 z )}{\eta (z ) \eta (3z)}\right]^{3}t\right)I_0(t)[K_0(t)]^3t\D t\\={}&\frac{\pi^3z}{4i}Z_{6,3}(z).\label{eq:JKKKK_YIKKK}\end{split}
\end{align} \end{corollary}\begin{proof}The definitions of $ J_0(x),x\in\mathbb R$ and $ Y_0(x),x\in(0,\infty)$ can be analytically continued to  $ J_0(z),z\in\mathbb C$, $ Y_0(z),z\in\mathbb C\smallsetminus(-\infty,0]$, through which one can define  the Hankel functions of zeroth order $ H_0^{(1)}(z)=J_0(z)+iY_0(z),z\in\mathbb C\smallsetminus(-\infty,0]$ and  $ H_0^{(2)}(z)=J_0(z)-iY_0(z),z\in\mathbb C\smallsetminus(-\infty,0]$.

Thus, the relation  $ J_0(x)= I_0(ix),x\in\mathbb R$ allows us to deduce \eqref{eq:JIKKK_Z63} from  \eqref{eq:IvKM_Z63}. The identity   $ \frac{2}{\pi i}K_0(y)=H_0^{(1)}(iy)=J_0(iy)+iY_0(iy),y>0$ reveals \eqref{eq:J_zZ63} and \eqref{eq:Y_zzZ63}  as natural consequences of \eqref{eq:IKvM_Z63}.

One can check that \eqref{eq:JKKKK_YIKKK} is an analytic continuation of  \eqref{eq:IvKM_IKvM_sum}, as in \cite[Proposition 5.1.4]{Zhou2017WEF}.  \end{proof}As in \cite{Broadhurst2016,Zhou2017WEF,Zhou2018ExpoDESY}, we introduce the following  cusp form of weight 6 and level 6:\begin{align}f_{6,6}(z):=
\frac{ [\eta (2 z) \eta (3 z)]^{9}}{[\eta (z)\eta (6 z)]^{3}}+\frac{ [\eta ( z) \eta (6 z)]^{9}}{[\eta (2z)\eta (3 z)]^{3}}=\frac{[Z_{6,3}(z)]^{2}}{2\pi i}\frac{\D X_{6,3}(z)}{\D z}.
\end{align}This cusp form featured prominently in Broadhurst's conjectures \cite[(142)--(146)]{Broadhurst2016} that related $ \IKM(a,8-a;1),a\in\{1,2,3,4\}$ to special values of the $L$-function: \begin{align}
L(f_{6,6},s):=\frac{1}{\Gamma(s)}\int_0^{\infty} f_{6,6}(iy)(2\pi y)^{s}\frac{\D y}{y}.
\end{align}In the next theorem, we recapitulate
from \cite[\S5]{Zhou2017WEF} our verification of Broadhurst's conjectures, along with some key simplifications.
\begin{theorem}[Broadhurst representations for 8-Bessel problems]\begin{enumerate}[leftmargin=*,  label=\emph{(\alph*)},ref=\emph{(\alph*)},
widest=a, align=left]\item We have {\allowdisplaybreaks\begin{align}
\IKM(4,4;1)={}&L(f_{6,6},3),\label{eq:IKM441_L}\\\IKM(3,5;1)={}&\frac{\pi^2}{4}L(f_{6,6},2),\label{eq:IKM351_L}\\\IKM(2,6;1)={}&\frac{\pi^{4}}{8}L(f_{6,6},1),\label{eq:IKM261_L}\\
\IKM(1,7;1)={}&\frac{\pi^4}{4}L(f_{6,6},2).\label{eq:IKM171_L}
\end{align}}\item We have the following identity:\begin{align}
7\pi^{2}L(f_{6,6},1)={}&36L(f_{6,6},3),\label{eq:ESM_L_value}
\end{align}which entails\begin{align}
14\IKM(2,6;1)=9\pi^{2}\IKM(4,4;1).
\end{align}\end{enumerate} \end{theorem}\begin{proof} \begin{enumerate}[leftmargin=*,  label={(\alph*)},ref=\emph{(\alph*)},
widest=a, align=left]\item As in \cite{Zhou2017WEF}, we refer to the Parseval--Plancherel identity for Hankel transforms, namely\begin{align}
\int_0^\infty \left[ \int_0^\infty J_0(xt) F(t)t\D t \right]\left[\int_0^\infty J_0(x\tau) G(\tau)\tau\D \tau\right]x\D x
= \int_0^\infty F(t)G(t)t\D t
\end{align}as ``Hankel fusion''.  Applying  Hankel fusions to \eqref{eq:JIKKK_Z63} and \eqref{eq:J_zZ63}, one can  verify \eqref{eq:IKM441_L}--\eqref{eq:IKM261_L} (cf.~\cite[(143)--(145)]{Broadhurst2016}, \cite[(1.2.7)--(1.2.9)]{Zhou2017WEF}, \cite[(23)--(25)]{Zhou2018ExpoDESY}).

Moreover, applying the Hilbert cancelation formula  \cite[(4.2.19)]{Zhou2017WEF} \begin{align}\int_0^\infty \left[ \int_0^\infty J_0(xt) F(t)t\D t \right]\left[\int_0^\infty Y_0(x\tau) F(\tau)\tau\D \tau\right]x\D x
=0\label{eq:Hilb_canc}
\end{align}to $ F(t)=I_0(t)[K_{0}(t)]^{3}$,   one can deduce \eqref{eq:IKM171_L} (cf.~\cite[(146)]{Broadhurst2016}, \cite[(1.2.8)]{Zhou2017WEF}, \cite[(24)]{Zhou2018ExpoDESY}) from   \eqref{eq:JIKKK_Z63} and \eqref{eq:JKKKK_YIKKK}. (See \cite[Lemma 4.2.4]{Zhou2017WEF} for the connection between \eqref{eq:Hilb_canc} and Hilbert transforms.) A comparison between  \eqref{eq:IKM351_L} and  \eqref{eq:IKM171_L} then leads us to a sum rule $ \IKM(1,7;1)=\pi^2\IKM(3,5;1)$ \cite[(148)]{Broadhurst2016}, which can also be verified by real-analytic  properties of Hilbert transforms \cite[Theorem 3.3]{HB1}, without invoking special $L$-values.
\item To begin, we note that the Hilbert cancelation formula can be extended to\begin{align}\begin{split}&
\int_0^\infty \left[ \int_0^\infty J_0(xt) F(t)t\D t \right]\left[\int_0^\infty Y_0(x\tau) G(\tau)\tau\D \tau\right]x\D x\\{}&
+\int_0^\infty \left[ \int_0^\infty J_0(xt) G(t)t\D t \right]\left[\int_0^\infty Y_0(x\tau)F(\tau)\tau\D \tau\right]x\D x=0,\end{split}
\end{align}for suitably regular $F  $ and $G$. Setting $ F(t)=I_0(t)[K_{0}(t)]^{3}$ and $ G(t)=[I_0(t)K_{0}(t)]^{2}$ in the equation above, while referring back to \eqref{eq:JIKKK_Z63}--\eqref{eq:JKKKK_YIKKK}, we obtain \begin{align}
\frac{\pi^{4}i}{6}\int_0^{i\infty}f_{6,6}(z)(1+18z^2)\D z-\frac{1}{2\pi}\IKM(2,6;1)=0.
\end{align}Using \eqref{eq:IKM261_L}, we can further  deduce that\begin{align}
\int_0^{i\infty}f_{6,6}(z)(7+72z^2)\D z=0,
\end{align}which is our goal. \qedhere\end{enumerate}\end{proof}
\begin{table}[t]\caption{A partial list of Vanhove operators $ \widetilde L_n$ for $u<0$, reformulated from   \cite[Table 1]{Vanhove2014Survey} so as to highlight the parity of each individual operator\label{tab:Vanhove_parity}}
\begin{small}\begin{align*}\begin{array}{c|l}\hline\hline n&\vphantom{\frac\int1}\widetilde{L}_n\\\hline 1&\vphantom{\frac\int1} \sqrt{u(u-4)}D^1\big[\sqrt{u(u-4)}D^0\big]\\2&D^1[u(u-1)(u-9)D^1]+(u-3)D^0\\3&D^1\big\{\sqrt{u^2(u-4)(u-16)}D^1\big[\sqrt{u^2(u-4)(u-16)}D^1\big]\big\}+\sqrt{u(u-8)}D^1\big[\sqrt{u(u-8)}D^0\big]\\4&D^2[u^{2}(u-1)(u-9)(u-25)D^2]+D^{1}[u (5 u^2-98 u+285)D^1]+(u-5)D^0\\\multirow{2}{*}{5}&D^2\big\{\sqrt{u^3(u-4)(u-16)(u-36)}D^1\big[\sqrt{u^3(u-4)(u-16)(u-36)}D^2\big]\big\}\\&+D^1\big\{\sqrt{u^2(5 u^2-168 u+1020)}D^1\big[\sqrt{u^2(5 u^2-168 u+1020)}D^1\big]\big\}+\sqrt{u(u-12)}D^1\big[\sqrt{u(u-12)}D^0\big]\\\hline\hline\end{array}\end{align*}\end{small}
\end{table}\begin{remark}Unlike our previous formulation of \cite[Theorems 5.2.1 and 5.2.2]{Zhou2017WEF},  the proof above requires no contour deformations.  Instead, it relies only on the properties of Hilbert transforms, and is essentially real-analytic. Thus, contrary to our  statement in the closing paragraph of \cite{HB1}, the sum rule  $ 14\IKM(2,6;1)=9\pi^{2}\IKM(4,4;1)$ is \textit{not} beyond the reach of Hilbert transforms.\eor\end{remark}
\begin{remark}
From the work of Yun \cite[Theorem 1.1.5]{Yun2015}, we know that  the Fourier coefficients $ a_p$ (for $p$ prime) in the weight-6 modular form $ f_{6,6}(z)=\sum_{m=1}^\infty a_me^{2\pi imz}$ encode the information for the $ 8$th symmetric power $ \Kl_{2}^{8}=\Sym^8(\Kl_2)$ of the Kloosterman sum $ \Kl_2(\mathbb F_{p},a)=\sum_{x\in\mathbb F_p^\times}e^{\frac{2\pi i}{p}\left(x+\frac{a}{x}\right)}$, averaged over $ a\in \mathbb F_p^\times$. Treating $  \Kl_2(\mathbb F_{p},a)$ as a finite-field analog\footnote{Prior to Broadhurst's study \cite{Broadhurst2016}, the idea of discretizing Bessel moments via Kloosterman sums appeared in Noam D. Elkies' post on MathOverflow \cite{Elkies2015} about the  integral $ \int_0^\infty x[J_0(x)]^5\D x$. For the connections between Elkies' question to quantum field theory, see \cite[Theorem 2.2.2]{Zhou2017WEF} and \cite[Theorem 4.1]{Zhou2017PlanarWalks}.} of the  Bessel function $ J_0(\sqrt{t}):=\frac{1}{2\pi i}\oint_{|z|=1}e^{\frac{i}{2}\left(z+\frac{t}{z_{\vphantom1}}\right)}\frac{\D z}{z}, t\in\mathbb C$, Broadhurst \cite{Broadhurst2016} and Broadhurst--Roberts \cite{BroadhurstRoberts2019} studied  $ \mathbb F_p^\times$-averages of   $ n$-th symmetric powers $ \Kl_{2}^{n}=\Sym^n(\Kl_2)$  as $p$-adic Bessel moments, leading to various conjectural evaluations of Feynman diagrams via $L$-functions for Kloosterman moments. Broadhurst and Roberts conjectured that these  $L$-functions satisfy a general  functional equation, which has been recently verified by Fres\'an--Sabbah--Yu \cite[Theorems 1.2 and 1.3]{FresanSabbahYu2018}.\eor\end{remark}
\subsection{Vanhove reflections and modular cancelation formulae}As seen from  Table~\ref{tab:Vanhove_parity}, the Vanhove operator $ \widetilde L_3^{\phantom1}=-\widetilde L_3^*$ is skew-symmetric (see \cite[Proposition 2.3]{Zhou2020BRquad} for a proof of the generic parity relation $ \widetilde L_m^{\phantom1}=(-1)^m \widetilde L_m^*$). If we define the commutator as $ [A,B]=AB-BA$, then we have\begin{align}\begin{split}&\left[\widetilde L_3,\frac{3\log(-u)-4\log(4-u)+\log(16-u)}{192}D^0\right]\\={}&3(uD^2+D^1)+\left[\frac{2}{(u-4)^2}+\frac{1}{3 (u-4)}+\frac{8}{(u-16)^2}+\frac{2}{3 (u-16)}\right]D^0.
\end{split}\label{eq:V_refl0}
\end{align}Now suppose that $ f,g\in\ker\widetilde L_3$ and define $ \langle f,g\rangle:=\int_{-\infty}^0 f(u)g(u)\D u$, then we can evaluate the integral\begin{align}
3\langle f,(uD^2+D^1)g\rangle+\int_{-\infty}^0 f(u)g(u)\left[\frac{2}{(u-4)^2}+\frac{1}{3 (u-4)}+\frac{8}{(u-16)^2}+\frac{2}{3 (u-16)}\right]\D u\label{eq:f_Dg}
\end{align}by collecting all  the boundary contributions from integration by parts. We will refer to this trick as a Vanhove reflection on the pair  $ f,g\in\ker\widetilde L_3$.

In what follows, we need some asymptotic  analyses in the $ u\to0^-$  and the $u\to-\infty $ regimes (Lemma \ref{lm:asympt0inf}), before fully elucidating the boundary contributions in Vanhove reflections (Theorem \ref{thm:ab_sum}).\begin{lemma}[Asymptotic behavior of certain off-shell Bessel moments]\label{lm:asympt0inf}The following relations hold true:\begin{align}
\IvKM(2,3;1|u)={}&\begin{cases}\frac{\pi^{2}}{16}\left[1 + \frac{u}{16} +O(u^{2})\right], & u\to0, \\[3pt]
-\frac{3\log^2\left( -\frac{1}{u} \right)}{4u}\left[1+O\left( \frac{1}{u} \right) \right], & u\to-\infty,\ \\
\end{cases}\label{eq:IvKM231_asympt}\\\IvKM(3,2;1|u)={}&\begin{cases}-\frac{\log\left( -\frac{u}{64} \right)}{8}\big(1+ \frac{u}{16}\big)+O(u), & u\to0^-, \\[3pt]
\frac{\log\left( -\frac{1}{u} \right)}{u}\left[1+O\left( \frac{1}{u} \right) \right], & u\to-\infty. \\
\end{cases}\label{eq:IvKM321_asympt}
\end{align}As a result, for \begin{align}
\ell(u):={}&\frac{3\log(-u)-4\log(4-u)+\log(16-u)}{192},\\\check Df(u):={}&\sqrt{u^2(u-4)(u-16)}D^{1}f(u),\end{align}we have asymptotic expansions {\allowdisplaybreaks\begin{align}
\check D^2\big[\ell(u)\IvKM(2,3;1|u)\big]={}&\begin{cases}O(u\log (-u)), & u\to0^{-}, \\
o(1), & u\to-\infty, \\
\end{cases}\label{eq:D2L231}\\
\check D\big[\ell(u)\IvKM(2,3;1|u)\big]
={}&\begin{cases}-\frac{\pi^{2}}{128}+O(u\log (-u)), & u\to0^{-}, \\
o(1), & u\to-\infty, \\
\end{cases}\label{eq:D1L231}\\\check D^2\IvKM(2,3;1|u)={}&\begin{cases}O(u), & u\to0, \\
o(1), & u\to-\infty, \\
\end{cases}\label{eq:D2_231}\\\check D\IvKM(2,3;1|u)={}&\begin{cases}O(u), & u\to0, \\
o(1), & u\to-\infty, \\
\end{cases}\label{eq:D1_231}
\\\check D^2\big[\ell(u)\IvKM(3,2;1|u)\big]={}&\begin{cases}-\frac{1}{4}+O(u\log^2(-u)), & u\to0^{-}, \\
o(1), & u\to-\infty, \\
\end{cases}\label{eq:D2L321}\\
\check D\big[\ell(u)\IvKM(3,2;1|u)\big]
={}&\begin{cases}\frac{1}{96} \log \left(-\frac{u^3}{2048}\right)+O(u\log^2(-u)), & u\to0^{-}, \\
o(1), & u\to-\infty, \\
\end{cases}\label{eq:D1L321}\\\check D^2\IvKM(3,2;1|u)={}&\begin{cases}O(u\log (-u)), & u\to0^{-}, \\
o(1), & u\to-\infty, \\
\end{cases}\label{eq:D2_321}\\\check D\IvKM(3,2;1|u)={}&\begin{cases}1+O(u\log (-u)), & u\to0^{-}, \\
o(1), & u\to-\infty, \\
\end{cases}\label{eq:D1_321}\end{align}}\end{lemma}\begin{proof}As $z\to i\infty $, we have $ u=-64X_{6,3}(z)\to0$ and [cf.~\eqref{eq:IvKM_Z63} and \eqref{eq:JIKKK_Z63}] \begin{align}
\IvKM(2,3;1|-64X_{6,3}(z))=\frac{\pi^{2}}{16}[1-4q+O(q^2)]
\end{align} for $ q=e^{2\pi iz }\to0$. Meanwhile, we have the $q$-expansion $ u=-64q+O(q^{2})$. This confirms the stated behavior of $\IvKM(2,3;1|u),u\to0$.

Let $ \widehat W_6:=\frac{1}{\sqrt{6}}\left(\begin{smallmatrix*}[r]0&-1\\6&0\end{smallmatrix*}\right)$ be the Chan--Zudilin involution \cite[\S2]{ChanZudilin2010} that sends $z$ to $ -\frac{1}{6z}$. Using the modular transformation $ \eta(-1/\tau)=\sqrt{\tau/i}\eta(\tau)$, one can verify that the modular function $ X_{6,3}(z)$ and the modular form $Z_{6,3}(z) $ satisfy the following transformation laws:
\begin{align}
X_{6,3}(\widehat W_6z)={}&\frac{1}{64X_{6,3}(z)},
\\Z_{6,3}(\widehat W_6z)={}&-48z^2Z_{6,3}(z)X_{6,3}(z),\end{align}for arbitrary $ z\in\mathfrak H$. Thus, we can perform  $q$-expansion on  $
\IvKM(2,3;1|-1/X_{6,3}(z))={}-3\pi ^{2}z^2Z_{6,3}(z)X_{6,3}(z)$, $z/i>0$ to deduce the leading order asymptotics of  $\IvKM(2,3;1|u),u\to-\infty$.

To verify \eqref{eq:IvKM321_asympt},    check the  $q$-expansions of  \eqref{eq:IvKM231_asympt} and $ \IvKM(3,2;1|-1/X_{6,3}(z))=\frac{2\pi z}{i}Z_{6,3}(z)X_{6,3}(z)$, $ z/i>0$, while exploiting the relation $ e^{2\pi iz}=-\frac{u}{64}+O(u^2)$ in the $ u\to0^-$ regime.

One can verify \eqref{eq:D2L231}--\eqref{eq:D1_321} through straightforward computations.   \end{proof}
\begin{theorem}[Exceptional sum rules via Vanhove reflections]\label{thm:ab_sum}If we define \begin{align}\varphi_{6,6}(z):={}&
f_{6,6}(z)\left\{\frac{2}{[-64X_{6,3}(z)-4]^2}+\frac{1}{3[-64X_{6,3}(z)-4]}\right\},\\\chi_{6,6}(z):={}&f_{6,6}(z)\left\{\frac{8}{[-64X_{6,3}(z)-16]^2} +\frac{2}{3[-64X_{6,3}(z)-16]}\right\},
\end{align}then we have \begin{align}
\IKM(2,6;3)={}&-\frac{\pi^5}{3i}\int_0^{i\infty}[\varphi_{6,6}(z)+\chi_{6,6}(z)]\D z=\frac{\IKM(2,6;1)}{72},\label{eq:IKM263_mod}\\-\frac{\pi^2}{192}+\IKM(3,5;3)={}&\frac{4\pi^{4}}{3}\int_0^{i\infty}[\varphi_{6,6}(z)+\chi_{6,6}(z)]z\D z=\frac{\IKM(3,5;1)}{72}+\frac{\pi^{2}}{1728},\label{eq:IKM353_mod}\\\frac{7\log2}{144}+\IKMh(4,4;3)={}&\frac{16\pi^{3}}{3}\int_0^{i\infty}[\varphi_{6,6}(z)+\chi_{6,6}(z)]z^{2}\D z=\frac{\IKM(4,4;1)}{72},\label{eq:IKM*443_mod}
\end{align} leading to \begin{align}
\dim_{\mathbb Q}\Span_{\mathbb Q}\{\IKM(2,6,2k+1)|k\in\mathbb Z_{\geq0}\}\leq2.
\end{align}\end{theorem}\begin{proof}We recall  from Table \ref{tab:Vanhove_parity} that $ \widetilde L_3=D^1\check D^2+ \sqrt{u(u-8)}D^1\big[ \sqrt{u(u-8)} D^{0} \big] $. Plugging \begin{align} f(u)=g(u)=\int_0^\infty J_0(\sqrt{-u}t)I_0(t)[K_0(t)]^3t\D t=\IvKM(2,3;1|u),\quad u<0\end{align}into \eqref{eq:f_Dg}, we have vanishing boundary contributions during integrations by parts:\begin{align}\begin{split}
\langle f,\widetilde L_3[\ell(u)g]\rangle
\underset{\eqref{eq:D2L231}}{\xlongequal{\eqref{eq:IvKM231_asympt}}}{}&-\left\langle\check  Df,D^1\check D\left[\ell(u)g\right]\right\rangle-\left\langle \sqrt{u(u-8)}D^1\big[ \sqrt{u(u-8)} f \big] ,\ell(u)g\right\rangle\\\underset{\eqref{eq:D1_231}}{\xlongequal{\eqref{eq:D1L231}}}{}&\left\langle\check  D^{2}f,D^1 \left[\ell(u)g\right]\right\rangle-\left\langle \sqrt{u(u-8)}D^1\big[ \sqrt{u(u-8)} f \big] ,\ell(u)g\right\rangle\\\underset{\eqref{eq:D2_231}}{\xlongequal{\eqref{eq:IvKM231_asympt}}}{}&-\langle\widetilde L_3f,\ell(u)g\rangle,\end{split}{}&,
\end{align}so \eqref{eq:f_Dg} must be equal to\begin{align}
\big\langle f,\big[\widetilde L_3,\ell(u)D^0\big]g\big\rangle
=-\langle\widetilde L_3f,\ell(u)g\rangle-\langle\ell(u)f,\widetilde L_3g\rangle=0.\label{eq:V_refl1}\end{align}Thus, a Vanhove reflection leads us to the first equality in \eqref{eq:IKM263_mod}, thanks to the Bessel differential equation   $ (uD^{2}+D^{1})\IvKM(2,3;1|u)=\frac{1}{4}\IvKM(2,3;3|u)$.

Performing a Vanhove reflection on $ f(u)=\IvKM(3,2;1|u),g(u)=\IvKM(2,3;1|u)$ [resp.~$ f(u)=g(u)=\IvKM(3,2;1|u)$], and carefully handling boundary contributions from integrations by parts, we can verify the first equality in \eqref{eq:IKM353_mod} [resp.~\eqref{eq:IKM*443_mod}]. Concretely speaking,  when  $ f(u)=\IvKM(3,2;1|u)$, $g(u)=\IvKM(2,3;1|u)$,  we have\begin{align}\begin{split}
\langle f,\widetilde L_3[\ell(u)g]\rangle
\underset{\eqref{eq:D2L231}}{\xlongequal{\eqref{eq:IvKM231_asympt},\eqref{eq:IvKM321_asympt}}}{}&-\left\langle\check  Df,D^1\check D\left[\ell(u)g\right]\right\rangle-\left\langle \sqrt{u(u-8)}D^1\big[ \sqrt{u(u-8)} f \big] ,\ell(u)g\right\rangle\\\underset{\eqref{eq:D1_321}}{\xlongequal{\eqref{eq:D1L231}}}{}&\frac{\pi^{2}}{128}+\left\langle\check  D^{2}f,D^1 \left[\ell(u)g\right]\right\rangle-\left\langle \sqrt{u(u-8)}D^1\big[ \sqrt{u(u-8)} f \big] ,\ell(u)g\right\rangle\\\underset{\eqref{eq:D2_321}}{\xlongequal{\eqref{eq:IvKM231_asympt}}}{}&\frac{\pi^{2}}{128}-\langle\widetilde L_3f,\ell(u)g\rangle,\end{split}
\end{align}which entails [cf.~the first equality in \eqref{eq:IKM353_mod}]\begin{align}\begin{split}&
\frac{\pi^{2}}{128}=\big\langle f,\big[\widetilde L_3,\ell(u)D^0\big]g\big\rangle
\\={}&\frac{3\IKM(3,5;3)}{2}+\int_{-\infty}^0 f(u)g(u)\left[\frac{2}{(u-4)^2}+\frac{1}{3 (u-4)}+\frac{8}{(u-16)^2}+\frac{2}{3 (u-16)}\right]\D u;\end{split}
\end{align}when $ f(u)=g(u)=\IvKM(3,2;1|u)$, we consider $\langle f_{1},f_{2}\rangle_\varepsilon=\int_{-\infty}^{-\varepsilon}f_{1}(u)f_{2}(u)\D u $ in the $ \varepsilon\to0^+$ regime, so that we have \begin{align}\begin{split}
\langle f, D^1\check D^2[\ell(u)g]\rangle _{\varepsilon}\underset{\eqref{eq:D2L321}}{\xlongequal{\eqref{eq:IvKM321_asympt}}}{}&\frac{\log\frac{\varepsilon}{64}}{32}-\left\langle\check  Df,D^1\check D\left[\ell(u)g\right]\right\rangle_\varepsilon+O(\varepsilon\log\varepsilon)\\\underset{\eqref{eq:D1_321}}{\xlongequal{\eqref{eq:D1L321}}}{}&\frac{\log\frac{\varepsilon}{64}}{32}-\frac{\log\frac{\varepsilon^{3}}{2048}}{96}+\left\langle\check  D^{2}f,D^1 \left[\ell(u)g\right]\right\rangle_\varepsilon+O(\varepsilon\log^{2}\varepsilon)\\\underset{\eqref{eq:D2_321}}{\xlongequal{\eqref{eq:IvKM321_asympt}}}{}&-\frac{7\log 2}{96}-\langle D^1\check D^2 f, \ell(u)g\rangle _{\varepsilon}+O(\varepsilon\log^{2}\varepsilon)\end{split}
\end{align}as well as \begin{align}
\left \langle f, \sqrt{u(u-8)}D^1\big[ \sqrt{u(u-8)} D^{0} \big]\big [\ell(u)g\big]\right\rangle\xlongequal{\eqref{eq:IvKM321_asympt}}-\left\langle \sqrt{u(u-8)}D^1\big[ \sqrt{u(u-8)} f \big] ,\ell(u)g\right\rangle,
\end{align} which can be combined into $ -\frac{7\log 2}{96}=\big\langle f,\big[\widetilde L_3,\ell(u)D^0\big]g\big\rangle$. Here, to fully connect this  to~the first equality in \eqref{eq:IKM*443_mod}, we also require the following observation\begin{align}
(uD^{2}+D^{1})\IvKM(3,2;1|u)=\frac{1}{4}\int_0^\infty J_0(\sqrt{-u}t)\left\{ [I_0(t)K_0(t)]^2-\frac{1}{4t^{2}} \right\}t^{3}\D t.\label{eq:IKM321''}
\end{align}    To verify the identity above, we need a  variation on \cite[(3.32)]{Zhou2018LaportaSunrise}:\begin{subequations}\begin{align}
\lim_{T\to\infty}\int_{0^+-iT}^{0^++iT}H_0^{(1)}(i\sqrt{-u}z)[H_0^{(1)}(z)H_0^{(2)}(z)]^{2}z\D z={}&0,\\\lim_{T\to\infty}\int_{0^+-iT}^{0^++iT}H_0^{(1)}(i\sqrt{-u}z)\left\{[H_0^{(1)}(z)H_0^{(2)}(z)]^{2}-\frac{4}{\pi^{2}z^2}\right\}z^{3}\D z={}&0,
\end{align}\end{subequations}where the contours close to the right. Spelling out the Hankel functions using the Bessel functions, we can turn the last pair of vanishing integrals into {\allowdisplaybreaks\begin{subequations}\begin{align}\begin{split}&
\int_0^\infty J_{0}(\sqrt{-u}t)[\pi I_{0}(t)K_0(t)]^{2}t\D t\\={}&\int_0^\infty J_{0}(\sqrt{-u}t)[K_0(t)]^{4}t\D t-2\pi\int_0^\infty Y_{0}(\sqrt{-u}t)I_{0}(t)[K_0(t)]^{3}t\D t,\label{eq:YIKKK_HT}\end{split}\\\begin{split}&
\int_0^\infty J_{0}(\sqrt{-u}t)\left\{[\pi I_{0}(t)K_0(t)]^{2}-\frac{\pi^{2}}{4t^{2}}\right\}t^{3}\D t\\={}&\int_0^\infty J_{0}(\sqrt{-u}t)[K_0(t)]^{4}t^{3}\D t-2\pi\int_0^\infty Y_{0}(\sqrt{-u}t)I_{0}(t)[K_0(t)]^{3}t^{3}\D t,\label{eq:YIKKK_HT''}\end{split}
\end{align}\end{subequations}}for $u<0$. Hitting the Bessel differential operator on the right hand side of \eqref{eq:YIKKK_HT} [which is equivalent to a combination of \eqref{eq:J_zZ63} and \eqref{eq:JKKKK_YIKKK}], we can deduce \eqref{eq:IKM321''} from \eqref{eq:YIKKK_HT''}.

Before verifying the second halves of  \eqref{eq:IKM263_mod}--\eqref{eq:IKM*443_mod}, we   transcribe the Chan--Zudilin base-change formulae \cite[(3.3)--(3.5)]{ChanZudilin2010} into  the following identities\begin{align}
2^{4}3^4\varphi_{6,6}(z)={}&\frac{[P_{-1/3}(1-2\alpha_3(z))]^4[1-2\alpha_3(z)]}{\pi i}\frac{\partial\alpha_3(z)}{\partial z},\label{eq:CZ_phi}\\\frac{3^{3}}{2}f_{6,6}(z)+2^{4}3^4\chi_{6,6}(z)={}&-\frac{[P_{-1/3}(1-2\alpha_3(2z))]^4[1-2\alpha_3(2z)]}{\pi i}\frac{\partial\alpha_3(2z)}{\partial z},\label{eq:CZ_chi}
\end{align}  for $ z/i>0$. Here, the Ramanujan cubic invariant (see \cite[Chap.~33, \S\S2--8]{RN5} and \cite[Chap.~9]{RLN2})\begin{align}
\alpha_{3}(z):=\left( \frac{[\eta(z/3)]^3}{3[\eta(3z)]^3}+1 \right)^{-3}=\left( \frac{[\eta(z)]^{12}}{27[\eta(3z)]^{12}}+1 \right)^{-1}\label{eq:alpha3_defn}
\end{align}maps the positive $ \I z$-axis bijectively to the open unit interval $ (0,1)$; the Legendre function of degree $ -1/3$  \cite[p.~22, (1.6.28)]{ET1} \begin{align}
P_{-1/3}(\cos\theta)=\frac{2}{\pi}\int_0^\theta \frac{\cos\frac{\beta}{3}}{\sqrt{\smash[b]{2(\cos\beta-\cos\theta)}}}\D\beta,\quad  \theta\in(0,\pi)
\end{align} satisfies \cite[Chap.~33, (5.24)]{RN5}\begin{align}
z=\frac{i P_{-1/3}(2\alpha_3(z)-1)}{\sqrt{3}P_{-1/3}(1-2\alpha_3(z))},\quad z/i>0.
\end{align}In view of the relations above, we arrive at the second equality in \eqref{eq:IKM263_mod}, after canceling out the contributions from the right-hand sides of \eqref{eq:CZ_phi} and \eqref{eq:CZ_chi}  as $ \int_{-1}^1x[P_{-1/3}(x)]^4\D x-\int_{-1}^1x[P_{-1/3}(x)]^4\D x=0$, and referring back to \eqref{eq:IKM261_L}. Similarly, the second equality in \eqref{eq:IKM*443_mod} follows from a trivial identity $ \int_{-1}^1x[P_{-1/3}(x)]^2[P_{-1/3}(-x)]^2\D x=0$ along with \eqref{eq:IKM441_L}. To deduce the second equality in  \eqref{eq:IKM353_mod}, we need both  \eqref{eq:IKM351_L} and a closed-form evaluation $ \int_{-1}^1 x[P_{-1/3}(x)]^3P_{-1/3}(-x)\D x=-\frac{9 \sqrt{3}}{4 \pi }$  \cite[(64)]{Zhou2013Pnu}.   \end{proof}

\begin{corollary}[Consequences of exceptional sum rules]The identities \eqref{eq:IKM171IKM173} and  \eqref{eq:nonlin_sum_8Bessel} are true. \end{corollary}
\begin{proof}

We recall the following relations for Crandall numbers \begin{align}\left\{\begin{array}{r@{\,=\,}l}
\pi^{2}\IKM (3,5;1)-\IKM (1,7;1)&0,\\\pi^{2}\IKM (3,5;3)-\IKM (1,7;3)&\frac{\pi^4}{2^7},\\\pi^{2}\IKM (3,5;5)-\IKM (1,7;5)&\frac{\pi^4}{2^8},\end{array}\right.\label{eq:Crandall_rel}
\end{align}which had arisen from numerical experiments of Broadhurst--Mellit (see  \cite[(7.10)]{BroadhurstMellit2016} or  \cite[(149) in Conjecture 5]{Broadhurst2016}) before being verified algebraically (see \cite[\S3.2]{HB1} or \cite[\S3]{Zhou2018ExpoDESY}). The first two equations of these allow us to deduce  \eqref{eq:IKM171IKM173}  from   \eqref{eq:IKM353_mod}.

The following determinant \begin{align}
\det\begin{pmatrix}\IKM (1,7;1)& \IKM (1,7;3) & \IKM (1,7;5) \\
\IKM (2,6;1)& \IKM (2,6;3) & \IKM (2,6;5) \\
\IKM (3,5;1)& \IKM (3,5;3) & \IKM (3,5;5) \\
\end{pmatrix}=\frac{5\pi^8}{2^{19}3}
\end{align}had been discovered numerically (see \cite[(7.11)]{BroadhurstMellit2016} or \cite[(163)]{Broadhurst2016}) before a mathematical proof was found \cite[\S4]{Zhou2017BMdet}.  According to \eqref{eq:Crandall_rel}, the determinant above must be equal to \begin{align}
\det\begin{pmatrix}\IKM (1,7;1)& \IKM (1,7;3) & \IKM (1,7;5) \\
\IKM (2,6;1)& \IKM (2,6;3) & \IKM (2,6;5) \\
0& \frac{\pi^{2}}{2^{7}} & \frac{\pi^{2}}{2^{8}} \\
\end{pmatrix}.
\end{align}    Then, we subtract $ \frac1{72}$ times the first column from the second column, while referring to \eqref{eq:IKM261IKM263}--\eqref{eq:IKM171IKM173}, so as to equate the last determinant with \begin{align}
\det\begin{pmatrix}\IKM (1,7;1)& -\frac{7 \pi ^4}{3456} & \IKM (1,7;5) \\
\IKM (2,6;1)& 0 & \IKM (2,6;5) \\
0& \frac{\pi^{2}}{2^{7}} & \frac{\pi^{2}}{2^{8}} \\
\end{pmatrix}.
\end{align}Therefore, the non-linear sum rule \eqref{eq:nonlin_sum_8Bessel} is true.\end{proof}

At present, we are not aware of any applications of the ``honorary Bessel moment'' $ \IKMh(4,4;3)$ to quantum field theory, but this quantity does play a r\^ole in the asymptotic analysis of a probability density function, as we describe below.

\begin{corollary}[Asymptotic behavior of Kluyver's $ p_7(x)$]Let $ p_n(x)=\int_0^\infty J_0(xt)[J_0(t)]^n xt\D t$ be Kluyver's probability density for the distance $x$ traveled by a rambler walking in the Euclidean plane, who takes $n$  steps of unit lengths, while aiming at uniformly distributed  directions. We have \begin{align}
\lim_{x\to1^-}\left[\frac{p_3(x)}{12\pi ^{2}x}+\frac{\D ^2}{\D x^2}\frac{p_{7}(x)}{35x}\right]=\frac{19L(f_{6,6},1)}{648\pi^{2}}+\frac{7\log 2}{72\pi^4}.\label{eq:p7_lim}
\end{align}In other words, the following asymptotic expansion applies to $ x\to1^-$:\begin{align}\begin{split}
\frac{p_7(x)}{35x}={}&\frac{L(f_{6,6},1)}{9\pi^{2}}\left[ 1-\frac{x-1}{4} +\frac{19(x-1)^{2}}{144}\right]-\frac{3 (x-1)^2}{32 \pi ^4}\\&+\frac{(x-1)^2}{16 \pi ^4}\log \frac{1-x}{2^{11/9}}+o((x-1)^2\log(1-x)).\end{split}\label{eq:p7_asympt}
\end{align}  \end{corollary}\begin{proof}From \cite[Theorem 4.2]{Zhou2017PlanarWalks}, we know that \begin{align}
\frac{p_{7}(x)}{35}=\frac{4}{\pi^6}\int_{0}^\infty I_{0} (xt)I_0(t)[K_0(t)]^6xt\D t-\frac{2}{\pi^4}\int_{0}^\infty I_{0} (xt)[I_0(t)]^3[K_0(t)]^4xt\D  t\label{eq:p7_IKM}
\end{align}holds for $ x\in[0,1]$. (This corrects \cite[(4.6)]{Zhou2017PlanarWalks}, where a   factor of $x$ went missing on the right-hand side.) Thus, it follows that (see \cite[Theorem 4.1, (4.3)]{Zhou2017PlanarWalks}; see also  \eqref{eq:IKM441_L}, \eqref{eq:IKM261_L} and \eqref{eq:ESM_L_value}  above)\begin{align}
\frac{p_{7}(1)}{35}=\frac{4}{\pi^6}\IKM(2,6;1)-\frac{2}{\pi^4}\IKM(4,4;1)=\frac{L(f_{6,6},1)}{9\pi^{2}}.
\end{align}Differentiating $ p_7(x)/x=\int_0^\infty J_0(xt)[J_0(t)]^7 t\D t$ under the integral sign and integrating by parts, we have \begin{align}
\left.\frac{\D }{\D x}\right|_{x=1}\frac{p_7(x)}{x}=-\int_0^\infty J_{1}(t)[J_0(t)]^7 t^{2}\D t=-\frac{p_{7}(1)}{4},\label{eq:p7'(1)}
\end{align}where $ J_1(x)=-\D J_0(x)/\D x$.

Meanwhile, using the Bessel differential equation, we can turn \eqref{eq:p7_IKM} into\begin{align}
\left( \frac{\D ^2}{\D x^2} +\frac{1}{x}\frac{\D }{\D x}\right)\frac{p_{7}(x)}{35x}=\frac{4}{\pi^6}\int_{0}^\infty I_{0} (xt)I_0(t)[K_0(t)]^6t^{3}\D t-\frac{2}{\pi^4}\int_{0}^\infty I_{0} (xt)[I_0(t)]^3[K_0(t)]^4t^{3}\D  t,
\end{align}for $ 0< x<1$. As we juxtapose the last equation with the following identity  (see \cite[Lemma 4.1.1]{Zhou2017WEF} or \cite[(3.3)]{Zhou2017PlanarWalks})\begin{align}
p_{3}(x)=\frac{6}{\pi^{2}}\int_0^\infty I_{0}(xt)I_{0}(t)[K_{0}(t)]^{2}xt\D t,\quad 0\leq x<1,
\end{align} we obtain \begin{align}
\lim_{x\to1^-}\left[\frac{p_3(x)}{12\pi ^{2}x}+\left( \frac{\D ^2}{\D x^2} +\frac{1}{x}\frac{\D }{\D x}\right)\frac{p_{7}(x)}{35x}\right]=\frac{4}{\pi^6}\IKM(2,6;3)-\frac{2}{\pi^4}\IKMh(4,4;3).
\end{align}Thanks to \eqref{eq:p7'(1)}, \eqref{eq:IKM441IKM*443} and \eqref{eq:IKM261IKM263}, this further reduces into\begin{align}
\lim_{x\to1^-}\left[\frac{p_3(x)}{12\pi ^{2}x}+\frac{\D ^2}{\D x^2}\frac{p_{7}(x)}{35x}\right]-\frac{p_{7}(1)}{140}=\frac{p_{7}(1)}{2520}+\frac{7\log 2}{72\pi^4},
\end{align} hence our claim in \eqref{eq:p7_lim}.

On the other hand, through asymptotic analysis of the explicit formula for $ p_3(x)$ \cite[(3.4)]{BSWZ2012}, we see that \begin{align}
\frac{p_3(x)}{12\pi ^{2}x}=\frac{1}{8\pi^{4}}\log\frac{4}{1-x}+O((x-1)\log(1-x)),\quad x\to1^-.
\end{align}Therefore, we arrive at \begin{align}
\lim_{x\to1^-}\left[-\frac{\log(1-x)}{8\pi^{4}}+\frac{\D ^2}{\D x^2}\frac{p_{7}(x)}{35x}\right]=\frac{19L(f_{6,6},1)}{648\pi^{2}}-\frac{11\log 2}{72\pi^4},
\end{align}which is compatible with \eqref{eq:p7_asympt}. \end{proof}\begin{remark}Comparing the right-hand side of \eqref{eq:p7_IKM} with \cite[Lemma 3.4(b)]{Zhou2020BRquad}, and setting \begin{align}\begin{split}
g(x,t):={}& I_0( xt)[K_0(t)]{}^5  \left\{[K_0(t)]{}^2-3 [\pi I_0(t)]^2\right\}\\{}&+K_0( xt)I_0(t) [K_0(t)]^4 \left\{7 [K_0(t)]{}^2-5 [\pi I_0(t)]^2\right\},\end{split}
\end{align}we see that\begin{align}
\frac{p_{7}(x)}{35}+\frac{2\log(1-x)}{\pi^{8}}\int_0^\infty g(x,t)t\D t
\end{align}equals a Taylor series expansion in powers of $ (x-1)^{\mathbb Z_{\geq0}}$, as $ x\to1^-$. Meanwhile, by \cite[Lemma 3.4(b)]{Zhou2020BRquad}, we have another Taylor expansion in the neighborhood of $x=1$, with leading order behavior \begin{align}
\int_0^\infty g(x,t)t\D t=-\frac{\pi^{4}}{32}(x-1)^2+O((x-1)^3).
\end{align}This explains the qualitative structure of the  right-hand side in \eqref{eq:p7_asympt}. \eor\end{remark}


\end{document}